\documentclass[12pt]{amsart}
\usepackage{amssymb}
\usepackage{amsfonts}
\usepackage{color}
\usepackage{hyperref}
\usepackage{blindtext}
\usepackage{mathtools}
\usepackage{bbm}
\allowdisplaybreaks

\theoremstyle{definition}
\newtheorem{theorem}{Theorem}[section]
\newtheorem{corollary}[theorem]{Corollary}
\newtheorem{lemma}[theorem]{Lemma}
\newtheorem{proposition}[theorem]{Proposition}

\newtheorem{remark}[theorem]{Remark}
\newtheorem{definition}[theorem]{Definition}

\newcommand{\mb}{\mathbb}
\newcommand{\mc}{\mathcal}
\newcommand{\eul}{\mathfrak}

\newcommand{\bou}{_{\scriptscriptstyle{\rm b}}}

\newcommand{\A}{\eul A}
\newcommand{\Ao}{{\eul A}_{\scriptscriptstyle 0}}

\newcommand{\Bo}{{\eul B}_{\scriptscriptstyle 0}}

\newcommand{\vp}{\varphi}

\newcommand{\Hil}{{\mc H}}
\newcommand{\C}{\mathrm{C}^{\ast}}

\newcommand{\D}{{\mc D}}

\def\x{\relax\ifmmode {\mbox{*}}\else*\fi}

\newcommand{\id}{\mathbbm{1}}
\newcommand{\ip}[2]{\left\langle{#1}|{#2}\right\rangle}

\newcommand{\ad}{^{\mbox{\scriptsize $\dag$}}}
\newcommand{\LDH}{{\mathcal L}\ad(\D,\Hil)}
\newcommand{\LDHpi}{{\mathcal L}\ad(\D_{\scriptscriptstyle\pi},\Hil_{\scriptscriptstyle\pi})}

\newcommand{\SSA}{{\mathcal S}_{\Ao}(\A)}

\newcommand{\rep}{{\mc R}(\A,\Ao)}
\newcommand{\crep}{{\mc R}_c(\A,\Ao)}

\def\H{{\mathcal H}}
\newcommand{\wmult}{\mbox{\raisebox{1pt}{$\scriptscriptstyle{
\square}$}}}

\numberwithin{equation}{section}

\numberwithin{equation}{section}

\begin{document}

\title[]{About tensor products of Hilbert quasi *-algebras and their representability properties}
\author{Maria Stella Adamo}
\address{Dipartimento di Matematica e Informatica, Universit\`a di Palermo, I-90123 Palermo, Italy}
\email{mariastella.adamo@community.unipa.it; msadamo@unict.it}

\begin{abstract} This note aims to investigate the tensor product of two given Hilbert quasi *-algebras and its properties.  The construction proposed in this note turns out to be again a Hilbert quasi *-algebra, thus interesting representability properties studed in \cite{AT} are maintained. Furthermore, if two functionals are representable and continuous respectively on the two Hilbert quasi *-algebras, then so is their tensor product.
\end{abstract}

\maketitle

\section{Introduction and basic notions}

The study of (locally convex) quasi *-algebras was initiated by G. Lassner in 1988 (see \cite{las,las1}) to give a rigorous solution to some problems coming from Quantum Statistical Mechanics. Since then, a wide range of literature appeared on this topic, both regarding applications and purely mathematical aspects of quasi *-algebras (see \cite{Ant1,Bag1,Bag2,Bag6,Frag2}). A special interest has been shown on the theory of representations on a suitable family of unbounded operators. In this context, new notions like full representability and *-semisimplicity have been introduced (e.g. \cite{Bag5,Frag2}). A central role is played by \textit{representable functionals}, i.e., those functionals that allow a GNS-like construction, in investigating structure properties of locally convex quasi *-algebras (see \cite{AT,Frag2,Trap1}).

The purpose of this note is to exhibit a tensor product construction for two given Hilbert quasi *-algebras and examine its properties. The interest in tensor products comes from physical phenomena, as they are employed to investigate two different physical systems as a joint one. In addition, very little is known about tensor products of unbounded operator algebras (e.g., \cite{adau,fiw,fiw1,hei}).

Hilbert quasi *-algebras are a particular subclass of locally convex quasi *-algebras arising as completions of Hilbert algebras with respect to the norm defined by their inner product. The main target is to build a new Hilbert quasi *-algebra that encodes the properties of the factors, by looking at them as Hilbert spaces. This would lead to defining the tensor product as a complex linear space that turns out to verify all the requirements for a quasi *-algebra on the tensor product of the *-algebras of the factors.

To obtain a tensor product Hilbert quasi *-algebra, it suffices to equip the tensor product quasi *-algebra with the inner product obtained as a combination of the two given in the factors and take its completion. In this way, we get a Hilbert quasi *-algebra, owning all the properties studied in \cite{AT}, in particular, it is fully representable and *-semisimple. Furthermore, given two representable and continuous functionals on factors, it is showed that their tensor product is still a representable and continuous functional. Hence, this highlights that the representability of functionals passes from factors to the tensor product.

The note is structured as follows. Firstly, preliminary notions about quasi *-algebras and representable functionals are recalled. In particular, for normed quasi *-algebras, in Section 2, we remind the notions of full representability, *-semisimplicity and the results obtained in the Banach case. Among them, for Hilbert quasi *-algebras, the characterization of representable and continuous functionals as elements of the Hilbert space will be employed to show the passing of representability for functionals to their tensor product. The latter fact has been proved in Section 3, in which we discuss the mentioned construction.

For the reader's convenience, we recall some preliminary notions for future use. Further details can be found in \cite{Ant1}.
\smallskip

\begin{definition}\label{quasi} A {\bf quasi *--algebra} $(\A, \Ao)$ is a pair consisting of a vector space $\A$ and a *-algebra $\Ao$ contained in $\A$ as a subspace and such that
\begin{itemize}
\item[(i)] $\A$ carries an involution $a\mapsto a^*$ extending the involution of $\Ao$;
\item[(ii)] $\A$ is  a bimodule over $\A_0$ and the module multiplications extend the multiplication of $\Ao$. In particular, the following associative laws hold:
\begin{equation}\notag \label{eq_associativity}
(xa)y = x(ay); \ \ a(xy)= (ax)y, \; \forall \, a \in \A, \,  x,y \in \Ao;
\end{equation}
\item[(iii)] $(ax)^*=x^*a^*$, for every $a \in \A$ and $x \in \Ao$.
\end{itemize}
\end{definition}

A quasi *-algebra $(\A, \Ao)$ is \emph{unital} if there is an element $\mathbbm{1}\in \Ao$, such that $a\mathbbm{1}=a=\mathbbm{1} a$, for all $a \in \A$; $\mathbbm{1}$ is unique and called the \emph{unit} of $(\A, \Ao)$.
\smallskip

A  quasi *-algebra $(\A,\Ao)$ is called a {\bf normed quasi *-algebra} if a norm
$\|\cdot\|$ is defined on $\A$ with the properties
\begin{itemize}
\item[(i)]$\|a^*\|=\|a\|, \quad \forall a \in \A$;
\item[(ii)] $\Ao$ is dense in $\A$;
\item[(iii)]for every $x \in \Ao$, the map $R_x: a \in \A \to ax \in \A$ is continuous in
$\A$.
\end{itemize}
{The continuity of the involution implies that
\begin{itemize}
\item[(iii')]for every $x \in \Ao$, the map $L_x: a \in \A \to xa \in \A$ is continuous in
$\A$.
\end{itemize}
}

\begin{definition}
If $(\A,\| \cdot \|) $ is a Banach space, we say that $(\A,\Ao)$ is a {\bf Banach quasi *-algebra}.\label{def} 
\end{definition}
The norm topology of $\A$ will be denoted by $\tau_n$. 
\smallskip

\smallskip

A special class of Banach quasi *-algebras is given by \textit{Hilbert quasi *-algebras}. Their interest comes from their rich structure, in which the norm arises from an inner product with certain properties.

\begin{definition}\label{Hilb_qalg}
	Let $\Ao$ be a *-algebra which is also a pre-Hilbert space with respect to the inner product $\ip{\cdot}{\cdot}$ such that:
	\begin{enumerate}
		\item the map $y\mapsto xy$ is continuous with respect to the norm defined by the inner product;
		\item $\ip{xy}{z}=\ip{y}{x^*z}$ for all $x,y,z\in\Ao$;
		\item $\ip{x}{y}=\ip{y^*}{x^*}$ for all $x,y\in\Ao$;
		\item $\Ao^2$ is total in $\Ao$.
	\end{enumerate}
If $\H$ denotes the Hilbert space completion of $\Ao$ with respect to the inner product $\ip{\cdot}{\cdot}$, then $(\H,\Ao)$ is called {\bf Hilbert quasi *-algebra}.
\end{definition}

\section{Representable functionals}
A convenient tool to study structure properties of quasi *-algebras is given by \textit{representable functionals}. Loosely speaking, those maps constitute a valid replacement for positive functionals in *-algebras for admitting a GNS-like triple.

In this section, we recall some properties related to representability of Banach quasi *-algebras. Among them, important notions are \textit{full representability} and \textit{*-semisimplicity}. For details, see \cite{AT,Ant1,Bag5, Frag2,Trap1}.
\smallskip

\begin{definition}\label{fun_repr} Let $(\A,\Ao)$ be a quasi *-algebra. A linear functional $\omega:\A\to\mathbb{C}$ satisfying 
\begin{itemize}
	\item[(L.1)]$\omega(x^*x) \geq 0, \quad \forall x \in \Ao$;
	\item[(L.2)]$\omega(y^*a^* x)= \overline{\omega(x^*ay)}, \quad\forall x,y \in \Ao, \forall a \in \A$;
	\item[(L.3)]$\forall a \in \A$, there exists $\gamma_a >0$ such that $$|\omega(a^*x)| \leq \gamma_a \omega(x^*x)^{1/2}, \quad \forall x\in \Ao.$$
\end{itemize}
is called {\bf representable} on the quasi *-algebra $(\A,\Ao)$.

The family of representable functionals on the quasi *-algebra $(\A,\Ao)$ is denoted by $\rep$.
\end{definition}

The term \textit{representable} is justified by the existence of a GNS-like triple of a *-representation $\pi_{\omega}$, a Hilbert space $\H_{\omega}$ and a map $\lambda_\omega$. In this context, we need a proper definition of *-representation of the quasi *-algebra $(\A,\Ao)$ and a family of operators on which we represent it.

\smallskip

Let $\Hil$ be a Hilbert space and let $\D$ be a dense linear subspace of $\Hil$. We denote by $\LDH$ the following family of closable operators
$$\LDH=\left\{X:\D\to\Hil:\mathcal{D}(X)=\D, \mathcal{D}(X^{\ast})\supset\D\right\}.$$
$\LDH$ is a $\mathbb{C}-$vector space with the usual sum and scalar multiplication. If we define the involution $\dagger$ and partial multiplication $\wmult$ as
$$X\mapsto X^{\dagger}\equiv X^{\ast}\upharpoonright_{\D}\quad\text{and}\quad X\wmult Y=X^{\dagger\ast}Y,$$
then $\LDH$ is a partial *-algebra defined in \cite{Ant1}.
\smallskip

\begin{definition}
	A {\bf *-representation} of a quasi *-algebra $(\A,\Ao)$ is a *-homomorphism $\pi:\A\to\LDHpi$, where $\D_{\scriptscriptstyle\pi}$ is a dense subspace of the Hilbert space $\Hil_{\scriptscriptstyle\pi}$, with the following properties:
	\begin{itemize}
		\item[(i)] $\pi(a^{\ast})=\pi(a)^{\dagger}$ for all $a\in\A$;
		\item[(ii)] if $a\in\A$ and $x\in\Ao$, then $\pi(a)$ is a left multiplier of $\pi(x)$ and $\pi(a)\wmult\pi(x)=\pi(ax)$.
	\end{itemize}
\end{definition}

A *-representation $\pi$ is 
\begin{itemize}
	\item {\em cyclic} if $\pi(\Ao)\xi$ is dense in $\Hil_{\scriptscriptstyle\pi}$ for some $\xi\in\D_{\scriptscriptstyle\pi}$;
	\item {\em closed} if $\pi$ coincides with its closure $\widetilde{\pi}$ defined in \cite[Section 2]{Trap1}.
\end{itemize}
If $(\A,\Ao)$ has a unit $\id$, then we suppose that $\pi(\id)=I_\D$, the identity operator of $\D$.
\smallskip

\begin{theorem}\cite{Trap1}\label{thm_repr}
	Let $(\A, \Ao)$ be a quasi *-algebra with unit
	$\id$ and let $\omega$ be a linear functional on $(\A, \Ao)$ that satisfies the conditions (L.1)-(L.3) of Definition \ref{fun_repr}.	Then, there exists a triple $(\pi_\omega,\lambda_\omega,\Hil_\omega)$ such that:
	\begin{itemize}
		\item $\pi_\omega$ is a closed cyclic *-representation ${\pi}_\omega$ of $(\A,\Ao)$, with cyclic vector $\xi_\omega$;
		\item $\lambda_\omega$ is a linear map of $\A$ into $\lambda_\omega(\Ao)=\D_{\pi_\omega}$, $\xi_\omega=\lambda_\omega(\id)$ and $\pi_\omega(a)\lambda_\omega(x)=\lambda_\omega(ax)$, for every $a\in\A$ and $x\in\Ao$;
		\item $ \omega(a)=\ip{{\pi}_\omega(a)\xi_\omega}{\xi_\omega}$, for every $a\in \A.$ 
	\end{itemize}
	This representation is unique up to unitary equivalence.
\end{theorem}

To every $\omega \in \rep$, we can associate the sesquilinear form $\varphi_{\omega}$ defined on $\Ao\times\Ao$ as
\begin{equation} \label{sesqu_ass} \vp_\omega(x,y):= \omega(y^*x), \quad x,y \in \Ao.\end{equation}
It is easy to see that
\begin{itemize}
\item[(i)] $\vp_\omega(x,x) \geq 0,$ for every $x \in \Ao$.
\item[(ii)] $\vp_\omega(xy,z)=\vp_\omega (y,x^*z)$ for every $x, y, z \in \Ao$.
\end{itemize}
\smallskip

If $(\A, \Ao)$ is a normed quasi *-algebra, we denote by $\crep$ the subset of $\rep$ consisting of continuous functionals.

As shown in \cite{Frag2} for locally convex quasi *-algebras, if $\omega\in \crep$, then the sesquilinear form $\vp_\omega$ defined in \eqref{sesqu_ass} is {\em closable}; that is, $\vp_\omega(x_n,x_n)\to 0$, for every sequence $\{x_n\}\subset \Ao$ such that $$\mbox{$\|x_n\|\to 0$ and $\vp_\omega(x_n-x_m, x_n-x_m)\to 0$}.$$

In this case, $\vp_\omega$  has a closed extension $\overline{\vp}_\omega$ to a dense domain $\D(\overline{\vp}_\omega)$ containing $\Ao$. For Banach quasi *-algebras this result can be improved.

\begin{proposition}\cite{AT}\label{prop1}
	Let $(\A, \Ao)$ be a Banach quasi *-algebra with unit $\id$,  $\omega \in \crep$ and $\vp_{\omega}$ the associated sesquilinear form on $\Ao \times \Ao$ defined as in \eqref{sesqu_ass}. Then $\D(\overline{\vp}_\omega)=\A$;  hence $\overline{\vp}_\omega$ is everywhere defined and bounded.
\end{proposition}

Consider now the set
$$\A_{\mc R}:= \bigcap_{\omega \in {\mc R}_c(\A,\Ao)}\D(\overline{\vp}_\omega).$$

If ${\mc R}_c(\A,\Ao)=\{0\}$, we put $\A_{\mc R}=\A$. Now set
$$\Ao^+:=\left\{\sum_{k=1}^n x_k^* x_k, \, x_k \in \Ao,\, n \in {\mb N}\right\}.$$
Then $\Ao^+$ is a wedge in $\Ao$ and we call the elements of $\Ao^+$ \emph{positive elements of} $\Ao$.
As in \cite{Frag2}, we call {positive elements of} $\A$ the members of $\overline{\Ao^+}^{\tau_n}$. We set $\A^+:=\overline{\Ao^+}^{\tau_n}$.
\smallskip

A linear functional on $\A$ is \textit{positive} if $\omega(a)\geq0$ for every $a\in\A^+$. 

\begin{definition} A family of positive linear functionals $\mc F$ on
$(\A[\tau_n], \Ao)$ is called {\bf sufficient} if for every $a \in
\A^+$, $a \neq 0$, there exists $\omega \in {\mc F}$ such that $\omega
(a)>0$.
\end{definition}

\begin{definition}\label{fully_rep} A normed quasi $^{\ast}$-algebra
$(\A[\tau_n],\Ao)$ is called {\bf fully representable} if ${\mc
R}_c(\A,\Ao)$ is sufficient and $\A_{\mc R}=\A$.
\end{definition}

We denote by $\SSA$ the family of all continuous sesquilinear forms $\Omega:\A\times\A\to\mathbb{C}$ such that
\begin{itemize}
	\item[(i)] $\Omega(a,a)\geq0$ for every $a\in\A$;
	\item[(ii)] $\Omega(ax,y)=\Omega(x,a^*y)$ for every $a\in\A$, $x,y\in\Ao$;
	\item[(iii)]$ |\Omega(a,b)|\leq \|a\|\|b\|$, for all $a, b\in \A.$
\end{itemize}

\begin{definition}\label{def1}
A normed quasi *-algebra $(\mathfrak{A}[\tau_n],\Ao)$ is called {\bf *-semi\-simple} if, for every $0\neq a\in\A$, there exists $\Omega\in\mathcal{S}_{\Ao}(\A)$ such that $\Omega(a,a)>0$.
\end{definition}

Proposition \ref{prop1} is useful to show the following result.

\begin{theorem}\cite{AT}\label{thm_fullrep_semis} Let $(\A, \Ao)$ be a Banach quasi *-algebra with unit $\id$. The following statements are equivalent. 
\begin{itemize}
\item[(i)]$\crep$ is sufficient.
\item[(ii)]$(\A,\Ao)$ is fully representable.
\end{itemize}
If the following condition of positivity $(P)$ 
\begin{equation}\notag\label{P} a\in\A\;\text{and}\;\omega(x^*ax)\geq0\;\;\forall\omega\in\mathcal{R}_c(\A,\Ao)\;\text{and}\;x\in\Ao\;\;\Rightarrow\;\;a\in\A^+
\end{equation}
holds, (i) and  (ii) are equivalent to the following
\begin{itemize}
\item[(iii)]$(\A,\Ao)$ is *-semisimple.
\end{itemize}
\end{theorem}

\begin{remark}\label{rem_thm}
	The condition $(P)$ is not needed to show the implication $(iii)\Rightarrow(ii)$ of Theorem \ref{thm_fullrep_semis}.
\end{remark}


For a Hilbert quasi *-algebra $(\H,\Ao)$, representable and continuous funciontals are in 1-1 correspondence with a certain family of elements in $\H$. 

\begin{definition}
 Let $(\H,\Ao)$ be a Hilbert quasi *-algebra. An element $\xi\in\H$ is called
 \begin{itemize}
 	\item[(i)] {\bf weakly positive} if the operator $L_{\xi}:\Ao\to\H$ defined as $L_{\xi}(x)=\xi x$ is positive.
 	\item[(ii)] {\bf bounded} if the operator $L_{\xi}:\Ao\to\H$ is bounded.
 \end{itemize}
The set of all weakly positive (resp. bounded) elements will be denoted as $\H^+_w$ (resp. $\H_{\bou}$), see \cite{AT,ct1}.
\end{definition}

The name \emph{weakly positive} is justified by the existence of a stronger notion of positivity introduced before (e.g. \cite{Frag2}). The comparison between them has been investigated in \cite{AT}.

\begin{remark}\label{bou}
	A Hilbert quasi *-algebra $(\H,\Ao)$ is always \emph{fully closable}, i.e., for every $\xi\in\H$ the densely defined operator $L_{\xi}:\Ao\to\H$ is closable. If $\xi\in\H_{\bou}$, then $\overline{L}_{\xi}$ is everywhere defined and bounded.
\end{remark}

\begin{theorem}\cite{AT}\label{Hrepc}
	Let $(\H,\Ao)$ be a Hilbert quasi *-algebra. Then $\omega\in\mathcal{R}_c(\H,\Ao)$ if, and only if, there exists a unique weakly positive bounded element $\eta\in\H$ such that
	$$\omega(\xi)=\ip{\xi}{\eta},\quad\forall\xi\in\H.$$
\end{theorem}

\section{Construction and properties of a tensor product Hilbert quasi *-algebra}
In this section, we will present a way to construct a tensor product quasi *-algebra of two given Hilbert quasi *-algebras. Our aim is to preserve structural properties of the factors in the tensor product, especially those related to representability. For further reading on algebraic tensor products see for instance \cite{cha}, for topological tensor products refer to \cite{deflo,fra,lau}. 
\smallskip

Let $(\H_1,\Ao)$ and $(\H_2,\Bo)$ be two given Hilbert quasi *-algebras. The first step in this construction consists of building up the algebraic tensor product. For convenience, we can assume that both the Hilbert quasi *-algebras are unital.

Observe that both $\H_1$ and $\H_2$ are bimodules, respectively over $\Ao$ and $\Bo$, thus we expect the tensor product to be a bimodule over a suitable *-algebra given by the tensor product *-algebra of $\Ao$ and $\Bo$.

\begin{definition}\label{ten_prod}\cite{hel2}
	Let $H, K, L$ be Hilbert spaces over the complex field and $\mathcal{R}:H\times K\to L$ a bilinear map. $\mathcal{R}$ is called a \textbf{Schimdt bilinear map} if for each total system $\{e_{\mu};\mu\in\Lambda_1\}$ in $H$ and $\{e_{\nu};\nu\in\Lambda_2\}$ in $K$ and each $z\in L$ we have $$\sum\{|\ip{\mathcal{R}(e_\mu,e_\nu)}{z}|^2:\mu\in\Lambda_1,\nu\in\Lambda_2\}<\infty.$$
\end{definition}

Note that the sum computed above depends only on the choice of $z$ and not on the chosen systems.

\begin{definition}\cite[Ch. 2]{hel2}
	Let $H, K, L$ be Hilbert spaces over the complex field and $\mathcal{R}:H\times K\to L$ a Schmidt bilinear map. The number
	$$\sup_{\|z\|_{\scriptscriptstyle L}\leq1}\sum\{|\ip{\mathcal{R}(e_\mu,e_\nu)}{z}|^2:\mu\in\Lambda_1,\nu\in\Lambda_2\}$$
	is said to be the \textbf{Schimdt norm} of the bilinear map $\mathcal{R}$.
\end{definition}

\begin{definition}\cite[Ch. 2]{hel2}
    Let $H, K, \Theta$ be Hilbert spaces over the complex field and $\Psi:H\times K\to \Theta$ a Schmidt bilinear map. The pair $(\Theta,\Psi)$ is a {\bf Hilbert tensor product of $H$ and $K$} if it has the universal property with respect to all Schmidt bilinear maps $\Psi':H\times K\to \Theta'$, that are Schimdt norm contractions, where $\Theta'$ is some complex Hilbert space. 
\end{definition}
	
The pair $(\Theta,\Psi)$ exists and it is unique, in the sense that if $(E',\Psi')$ is another tensor product of the Hilbert spaces $H,K$ as in Definition \ref{ten_prod}, then there is an algebraic isomorphism $i:\Theta\to\Theta'$ such that $i\circ\Psi=\Psi'$. Hence, the unique Hilbert tensor product will be denoted by $H\widehat{\otimes}K$.

\begin{theorem}\label{cts_op}\cite[Ch. 2]{hel2}
	Let $S:K_1\to K_2$ and $T:L_1\to L_2$ be bounded operators between Hilbert spaces. Then there exists a unique bounded operator $S\widehat{\otimes} T:K_1\widehat{\otimes} L_1\to K_2\widehat{\otimes} L_2$ such that $(S\widehat{\otimes} T)(x\otimes y)=S(x)\otimes T(y)$ for all $x\in K_1$, $y\in L_1$. Moreover $\|S\widehat{\otimes} T\|=\|S\|\,\|T\|$.
\end{theorem}

Having at our hands these notions, we construct the tensor product explicitly and show that it is a quasi *-algebra. 

$\H_1\otimes\H_2$ is built as the vector space tensor product of $\H_1$ and $\H_2$ over $\mathbb{C}$. $\Ao\otimes\Bo$ can be naturally regarded as a subspace of $\H_1\otimes\H_2$, since $\Ao$ and $\Bo$ are subspaces of $\H_1$ and $\H_2$ respectively. $\Ao$ and $\Bo$ are also *-algebras, thus their vector space tensor product $\Ao\otimes\Bo$ becomes a *-algebra if we define the product of two elements $z=\sum_{i=1}^nx_i\otimes y_i$, $z'=\sum_{j=1}^mx'_j\otimes y'_j$ in $\Ao\otimes\Bo$ as follows
$$zz':=\sum_{i=1}^n\sum_{j=1}^mx_ix_j'\otimes y_iy_j'$$
and an involution $\ast$ as 
$$z^*:=\sum_{i=1}^nx^*_i\otimes y^*_i.$$
Note that the product and the involution are well-defined (see \cite[pp. 188,189]{murphy}).
\smallskip

A general element $\zeta\in\H_1\otimes\H_2$ is of the form
$$\zeta=\sum_{i=1}^n\xi_i\otimes\eta_i,\quad\xi_i\in\H_1,\eta_i\in\H_2,\quad 1\leq i \leq n.$$
The $\mathbb{C}$-vector space $\H_1\otimes\H_2$ becomes a bimodule over the *-algebra $\Ao\otimes\Bo$ defining the module actions and the involution component-wise, i.e., for every $x\otimes y\in\Ao\otimes\Bo$, $\xi\otimes \eta\in\H_1\otimes\H_2$
\begin{align}\notag
&(x\otimes y)(\xi\otimes \eta)=x\xi\otimes y\eta\quad\text{and}\quad(\xi\otimes \eta)(x\otimes y)=\xi x\otimes \eta y;\\
\notag&(\xi\otimes \eta)^*=\xi^*\otimes \eta^*.
\end{align}
It is easily shown that the preceding operations are well defined, extending those defined in $\Ao\otimes\Bo$ (see \cite[pp. 188-189]{murphy} and \cite[Lemma 1.4, p. 361]{mal}).

All the requirements of Definition \ref{quasi} are easily verified by the properties of the quasi *-algebras $(\H_1,\Ao)$ and $(\H_2,\Bo)$. Thus, we conclude that \textit{$(\H_1\otimes\H_2,\Ao\otimes\Bo)$ is a quasi *-algebra}.
\smallskip

The second step consists of providing the quasi *-algebra $\H_1\otimes\H_2$ with a suitable inner product in the way it becomes a \textit{Hilbert} quasi *-algebra. 

If we denote by $\ip{\cdot}{\cdot}_1$ and $\ip{\cdot}{\cdot}_2$ the inner products of $\H_1$ and $\H_2$ respectively, the most natural choice is given by
\begin{equation}\label{inn_prod}
\ip{\zeta}{\zeta'}:=\sum_{i=1}^n\sum_{j=1}^{m}\ip{\xi_i}{\xi_j'}_1\ip{\eta_i}{\eta_j'}_2,\quad\forall\zeta,\zeta'\in\H_1\otimes\H_2,
\end{equation}
where $\zeta=\sum_{i=1}^n\xi_i\otimes\eta_i$ and $\zeta'=\sum_{j=1}^m\xi'_j\otimes\eta_j'$. By Lemma 1.1 in \cite[Ch. IV]{tak}, the map defined in \eqref{inn_prod} is a well-defined inner product on $\H_1\otimes\H_2$.

The norm $\|\cdot\|_h$ induced by the inner product in \eqref{inn_prod} is a \textit{cross-norm}, i.e., if $\zeta=\xi\otimes\eta$ with $\xi\in\H_1$, $\eta\in\H_2$ is an elementary tensor, then $\|\zeta\|_h=\|\xi\otimes\eta\|_h=\|\xi\|_1\,\|\eta\|_2$.
\smallskip

Taking the completion $\H_1\widehat{\otimes}^h\H_2$ of $\H_1\otimes^h\H_2$ with respect to the topology $h$ corresponding to $\|\cdot\|_h$, we obtain a Hilbert space, which is the Hilbert tensor product according to Definition \ref{ten_prod}. The tensor map $\Theta: \H_1\times\H_2\to\H_1\widehat{\otimes}^h\H_2$ is a Schmidt bilinear map. 

We want to show that $\H_1\widehat{\otimes}^h\H_2$ can be regarded as a Hilbert quasi *-algebra over $\Ao\otimes\Bo$.

Conditions (2) - (4) of Definition \ref{Hilb_qalg} are obviously verified by the properties of Hilbert quasi *-algebras $(\H_1,\Ao)$ and $(\H_2,\Bo)$. To show (1), we have to prove the continuity of the left multiplication operators
$$L_z:\H_1\widehat{\otimes}^h\H_2\to\H_1\widehat{\otimes}^h\H_2\quad\text{defined as}\quad\zeta\mapsto L_z(\zeta)=\zeta z$$
for every $z\in\Ao\otimes\Bo$. Without loss of generality, we can assume that $z=x\otimes y$, for $x\otimes y\in\Ao\otimes\Bo$. It is enough to show that the operator $L_z$ restricted to $\H_1\otimes\H_2$ is continuous.

The restriction of $L_{x\otimes y}$ to $\H_1\otimes\H_2$ coincides with the operator $L_x\otimes L_y$, thus, by Theorem \ref{cts_op}, it is continuous since the operators $L_x$, $L_y$ are continuous.

What remains to be shown is that in fact the completion $\H_1\widehat{\otimes}^h\H_2$ is the same as $\Ao\widehat{\otimes}^h\Bo$. For our aim, it is enough to show that $\Ao\otimes\Bo$ is $h$-dense in $\H_1\otimes\H_2$. Without loss of generality, we show the statement only for elementary tensors in $\H_1\otimes\H_2$.

Let $\xi\otimes\eta$ an elementary tensor in $\H_1\otimes\H_2$. By the properties of Hilbert quasi *-algebra of $\H_1$ and $\H_2$, there exist sequences $\{x_n\}$ and $\{y_n\}$ in $\Ao$ and $\Bo$ respectively, such that $\|x_n-\xi\|_1\to0$ and $\|y_n-\eta\|_2\to0$. By the cross-norm property of $\|\cdot\|_h$, the sequence $\{x_n\otimes y_n\}$ in $\Ao\otimes\Bo$ approximates $\xi\otimes\eta$.

Since $\Ao\otimes\Bo$ is dense in $\H_1\otimes\H_2$, we can deduce that
$$\Ao\widehat{\otimes}^h\Bo\cong\H_1\widehat{\otimes}^h\H_2.$$

In conclusion, we have proved the following

\begin{theorem}\label{tensor}
	Let $(\H_1,\Ao)$ and $(\H_2,\Bo)$ be unital Hilbert quasi *-algebras. Then the completion $\Ao\widehat{\otimes}^h\Bo\cong\H_1\widehat{\otimes}^h\H_2$ of $\Ao\otimes\Bo$ with respect to the inner product \eqref{inn_prod} is a unital Hilbert quasi *-algebra.
\end{theorem}

\begin{definition}
Let $(\H_1,\Ao)$, $(\H_2,\Bo)$ be unital Hilbert quasi *-alge\-bras. Then the Hilbert quasi *-algebra $(\H_1\widehat{\otimes}^h\H_2,\Ao\otimes\Bo)$ constructed in Theorem \ref{tensor} will be called {\bf tensor product Hilbert quasi *-algebra}.
\end{definition}

For the tensor product Hilbert quasi *-algebra $(\H_1\widehat{\otimes}^h\H_2,\Ao\otimes\Bo)$ obtained in Theorem \ref{tensor} apply all the representability results obtained for Hilbert quasi *-algebras in \cite{AT}.

If $\Omega:\H_1\widehat{\otimes}^h\H_2\to\mathbb{C}$ is a representable and continuous functional, then by Proposition \ref{prop1}, the domain of the sesquilinear form $\overline{\varphi}_{\Omega}$ is the whole $\H_1\widehat{\otimes}^h\H_2$.

Moreover, \textit{Hilbert quasi *-algebras are always *-semisimple}. Hence, $(\H_1\widehat{\otimes}^h\H_2,\Ao\otimes\Bo)$ is *-semisimple, and, by Theorem \ref{thm_fullrep_semis} and Remark \ref{rem_thm}, it is also fully representable.
\smallskip

We conclude then with the following

\begin{corollary}\label{fr_ss}
	Let $(\H_1,\Ao)$ and $(\H_2,\Bo)$ be unital Hilbert quasi *-algebras. Then the tensor product Hilbert quasi *-algebra $(\H_1\widehat{\otimes}^h\H_2,$ $\Ao\otimes\Bo)$ is *-semisimple and fully representable.
\end{corollary}

By Remark \ref{bou}, the Hilbert quasi *-algebra $(\H_1\widehat{\otimes}^h\H_2,\Ao\otimes\Bo)$ is also \textit{fully closable}. Note that this property can be now derived directly by Corollary \ref{fr_ss}.

Corollary \ref{fr_ss} highlights the important fact that the representability properties are \textit{maintained} passing from the factors to the tensor product. Moreover, given a couple of representable and continuous functionals defined on the factors, the tensor product of these functionals turns out to be \textit{always} representable and continuous, as shown in Proposition \ref{repc_tens} in the next section. 
\smallskip

The same question about representability has been investigated in a more general setting of Banach quasi *-algebras in \cite{AF}.

\section{Representability for tensor product Hilbert quasi *-algebras}
In this section, we investigate how representability of functionals passes to their tensor product. Since for Hilbert quasi *-algebras representable and continuous functionals have been characterized in terms of weakly positive bounded elements, we first prove a lemma about the tensor product of these elements.

\begin{lemma}\label{wb}
Let $(\H_1\widehat{\otimes}^h\H_2,\Ao\otimes\Bo)$ be the tensor product Hilbert quasi *-algebra of $(\H_1,\Ao)$ and $(\H_2,\Bo)$, as in Theorem \ref{tensor}. Then, we have the following inclusions
\begin{itemize}
	\item[(i)] $(\H_1)^+_w\otimes(\H_2)_w^+\subseteq(\H_1\widehat{\otimes}^h\H_2)^+_w$.
	\item[(ii)] $(\H_1)_{\bou}\otimes(\H_2)_{\bou}\subseteq(\H_1\widehat{\otimes}^h\H_2)_{\bou}$.
\end{itemize}
\end{lemma}
\begin{proof}
Let $\eta_1\in(\H_1)^+_w$ and $\eta_2\in(\H_2)^+_w$. Then, the multiplication operators $R_{\eta_1}:\Ao\to\H_1$ and $R_{\eta_2}:\Bo\to\H_2$ are densely defined and positive. Let $R'_{\eta_1}$ and $R'_{\eta_2}$ be positive self-adjoint extensions of $R_{\eta_1}$ and $R_{\eta_2}$ respectively. We have $R_{\eta_1\otimes\eta_2}=R_{\eta_1}\otimes R_{\eta_2}$ on $\Ao\otimes\Bo$.

The tensor product $\eta_1\otimes\eta_2$ is again a weakly positive element. Indeed, let $z=\sum_{i=1}^nx_i\otimes y_i$ be in $\Ao\otimes\Bo$, then

\begin{align*}
\ip{z}{R_{\eta_1\otimes\eta_2}z}&=\ip{\sum_{i=1}^nx_i\otimes y_i}{(R_{\eta_1}\otimes R_{\eta_2})\left(\sum_{j=1}^nx_j\otimes y_j\right)}\\
&=\sum_{i,j=1}^n\ip{x_i}{R_{\eta_1}x_j}_1\ip{y_i}{R_{\eta_2}y_j}_2\\
&=\sum_{i,j=1}^n\ip{x_i}{R'_{\eta_1}x_j}_1\ip{y_i}{R'_{\eta_2}y_j}_2\\
&=\sum_{i,j=1}^n\ip{(R'_{\eta_1})^{\frac12}x_i}{(R'_{\eta_1})^{\frac12}x_j}_1\ip{(R'_{\eta_2})^{\frac12}y_i}{(R'_{\eta_2})^{\frac12}y_j}_2\\
&=\sum_{i,j=1}^n\ip{x'_i}{x'_j}_1\ip{y'_i}{y'_j}_2,
\end{align*}

where $x'_i=(R'_{\eta_1})^{\frac12}x_i\in\H_1$ and $y'_i=(R'_{\eta_2})^{\frac12}y_i\in\H_2$, for every $i=1,\ldots,n$. Applying the Gram-Schmidt orthogonalization, we can assume that $\{x'_i\}_{i=1}^n$ are orthogonal. Hence,
$$\ip{z}{R_{\eta_1\otimes\eta_2}z}=\sum_{i,j=1}^n\ip{x'_i}{x'_j}_1\ip{y'_i}{y'_j}_2=\sum_{i=1}^n\ip{x'_i}{x'_i}_1\ip{y'_i}{y'_i}_2\geq0.$$

Let $\chi_1\in(\H_1)_{\bou}$ and $\chi_2\in(\H_2)_{\bou}$. Then $\chi_1\otimes\chi_2$ is also bounded. Indeed, the left multiplication operator $L_{\chi_1\otimes\chi_2}:\H_1\otimes\H_2\to\H_1\widehat{\otimes}^h\H_2$ is equal to $L_{\chi_1}\otimes L_{\chi_2}$. Thus, it is bounded by Theorem \ref{cts_op} and continuity of $L_{\chi_1}$ and $L_{\chi_2}$.
\end{proof}


\begin{proposition}\label{repc_tens}
Let $(\H_1,\Ao)$, $(\H_2,\Bo)$ be unital Hilbert quasi *-algebras. Let $(\H_1\widehat{\otimes}^h\H_2,\Ao\otimes\Bo)$ be the tensor product Hilbert quasi *-algebra. If $\omega_1$ and $\omega_2$ are representable and continuous functionals on $\H_1$ and $\H_2$ respectively, then $\omega_1\otimes\omega_2$ extends continuously to a representable and continuous functional on $\H_1\widehat{\otimes}^h\H_2$.
\end{proposition}

\begin{proof}
Let $\omega_1$ and $\omega_2$  be representable and continuous functionals on $\H_1$ and $\H_2$ respectively. By Theorem \ref{Hrepc}, there exist weakly positive bounded elements $\chi_1\in\H_1$ and $\chi_2\in\H_2$ such that
\begin{align}\label{thm}&\omega_1(\eta_1)=\ip{\eta_1}{\chi_1}_1,\quad\forall\eta_1\in\H_1,\\
\label{thm_2}&\omega_2(\eta_2)=\ip{\eta_2}{\chi_2}_2,\quad\forall\eta_2\in\H_2.
\end{align}
Define now $\omega_1\otimes\omega_2$ on $\H_1\otimes^h\H_2$ as 
\footnotesize{$$\omega_1\otimes\omega_2\left(\sum_{i=1}^n\xi_i\otimes\eta_i\right):=\sum_{i=1}^n\omega_1(\xi_i)\omega_2(\eta_i),\quad\forall\sum_{i=1}^n\xi_i\otimes\eta_i\in\H_1\otimes^h\H_2.$$}
\normalsize$\omega_1\otimes\omega_2$ is a well defined linear map on $\H_1\otimes^h\H_2$. By \eqref{thm} and \eqref{thm_2},
\small{$$\omega_1\otimes\omega_2\left(\sum_{i=1}^n\xi_i\otimes\eta_i\right)=\sum_{i=1}^n\ip{\xi_i}{\chi_1}_1\ip{\eta_i}{\chi_2}_2=\ip{\sum_{i=1}^n\xi_i\otimes\eta_i}{\chi_1\otimes\chi_2}_h$$}
\normalsize for all $\textstyle\sum_{i=1}^n\xi_i\otimes\eta_i\in\H_1\otimes^h\H_2$. Hence, $\omega_1\otimes\omega_2$ is continuous and denote by $\Omega$ its continuous extension to $\H_1\widehat{\otimes}^h\H_2$.
\smallskip

By Lemma \ref{wb}, $\chi_1\otimes\chi_2$ is weakly positive and bounded on $\H_1\widehat{\otimes}^h\H_2$. Therefore, $\Omega$ is representable and continuous. Indeed, if $\left\{\sum_{i=1}^{k_n}x^i_n\otimes y^i_n\right\}$ is a sequence of elements in $\Ao\otimes\Bo$ approximating $\psi\in\H_1\widehat{\otimes}^h\H_2$, then $\Omega$ is given by
\small{$$\Omega(\psi):=\lim_{n\to+\infty}\omega_1\otimes\omega_2\left(\sum_{i=1}^{k_n}x^i_n\otimes y^i_n\right),\quad\forall\psi\in\H_1\widehat{\otimes}^h\H_2.$$}
\normalsize By the properties of $\omega_1$ and $\omega_2$, 
\small{\begin{align*}
\Omega(\psi)&=\lim_{n\to+\infty}\omega_1\otimes\omega_2\left(\sum_{i=1}^{k_n}x^i_n\otimes y^i_n\right)\\
&=\lim_{n\to+\infty}\ip{\sum_{i=1}^{k_n}x^i_n\otimes y^i_n}{\chi_1\otimes\chi_2}\\
&=\ip{\psi}{\chi_1\otimes\chi_2},
\end{align*}}
\normalsize for all $\psi\in\H_1\widehat{\otimes}^h\H_2$. By Theorem \ref{Hrepc} we get the representability of $\Omega$ on $\H_1\widehat{\otimes}^h\H_2$.
\end{proof}

\section*{Final remarks and open problems}
A key role in looking at structure properties of locally convex quasi *-algebras is played by representable (and continuous) functionals (see \cite{AT,Bag5,Frag2,Trap1}). In this paper we constructed a new Hilbert quasi *-algebra as the tensor product of two given ones and then we studied its properties, based on the results obtained in \cite{AT}. Nevertheless, many interesting questions remained open.

It would be of interest to understand which is the relationship between bounded elements in the factors $(\H_1,\Ao)$ and $(\H_2,\Bo)$ and those in the tensor product $(\H_1\widehat{\otimes}^h\H_2,\Ao\otimes\Bo)$. A similar question can be posed for weakly positive elements.

In (ii) of Lemma \ref{wb}, we have shown that the tensor product of $(\H_1)_{\bou}$ and $(\H_2)_{\bou}$ is contained in the C*-algebra $(\H_1\widehat{\otimes}^h\H_2)_{\bou}$ (see also \cite[Proposition 4.9]{AT}). It is unknown if it is isomorphic of a certain C*-completion of 
the tensor product $(\H_1)_{\bou}\otimes(\H_2)_{\bou}$. 

Answering these questions concerns the study of representable and continuous functionals on the tensor product Hilbert quasi *-algebra. In particular, if $\Omega\in\mathcal{R}_c(\H_1\widehat{\otimes}^h\H_2,\Ao\otimes\Bo)$, the restrictions $\omega_1$ on $\H_1$ and $\omega_2$ on $\H_2$ of $\Omega$ belong to $\mathcal{R}_c(\H_1,\Ao)$ and $\mathcal{R}_c(\H_2,\Bo)$ respectively. By Proposition \ref{repc_tens}, the tensor product $\omega_1\otimes\omega_2$ extends to a functional $\widetilde{\Omega}$ in $\mathcal{R}_c(\H_1\widehat{\otimes}^h\H_2,\Ao\otimes\Bo)$. A natural question to ask is under which conditions this extension corresponds to $\Omega$. 

The construction given in Theorem \ref{tensor} could be useful to exhibit other examples of Hilbert quasi *-algebras $(\H,\Ao)$ for which $\mathcal{R}_c(\H,\Ao)$ is equal to $\mathcal{R}(\H,\Ao)$. By Propositions 5.1 and 5.2 of \cite{AT}, this is true for the Hilbert quasi *-algebras $(L^2(I,d\lambda),\mathcal{C}(I))$ and $(L^2(I,d\lambda),L^\infty(I,d\lambda))$, where $I$ is a compact interval of the real line and $\lambda$ is the Lebesgue measure.

We leave these questions about tensor product of Hilbert quasi *-algebras for future papers, with the thought that they can contribute to give a better understanding of locally convex quasi *-algebras and their tensor products. 
\smallskip

{\bf{Acknowledgment:} }
The author is grateful to the Organizers of the 27th International Conference in Operator Theory for this wonderful conference and the West University of Timisoara for its hospitality. The author also wishes to thank prof. M. Fragoulopoulou for interesting discussions and suggestions.

\end{document}